\documentclass[reqno,a4paper, roman, 12pt]{amsart}

\usepackage{bm}
\usepackage{amsmath}
\usepackage{amsfonts}
\usepackage{amssymb}
\usepackage{amsthm}
\usepackage{graphicx}
\usepackage{color}
\usepackage{enumitem}
\usepackage{url}
\usepackage{a4wide}
\theoremstyle{plain}

\newtheorem{theorem}{Theorem}[section]

\newtheorem{remark}[theorem]{Remark}

\numberwithin{equation}{section}
\newcommand{\tr}{\mathop{\mathrm{tr}}}

\newcommand{\diag}{\mathop{\mathrm{diag}}}
\newcommand{\rank}{\mathop{\mathrm{rank}}}

\begin{document}
\title[On Wishart distributions and their infinite divisibility]{On the parameter domain of Wishart distributions and their infinite divisibility}
\author{Eberhard Mayerhofer}
\address{Vienna Institute of Finance,
University of Vienna and Vienna University of Economics and Business, Heiligenst\"adterstrasse 46-48,
1190 Vienna, Austria}
\email{eberhard.mayerhofer@vif.ac.at}
\begin{abstract}
A complete characterization of Wishart distributions on the cones of
positive semi-definite matrices is provided in terms of a
description of their maximal parameter domain. This result is new in
that also degenerate scale parameters are included. For such cases,
the standard constraints on the range of the shape parameter may be
relaxed. Furthermore, the infinitely divisible Wishart distributions
are revealed as suitable transformations and embeddings of one
dimensional gamma distributions. This note completes the findings of
L\'evy (1937) concerning infinite divisibility and Gindikin (1975)
regarding the existence issue.
\end{abstract}
\maketitle
\section{Introduction}
The Wishart distribution $\Gamma(p;\sigma)$ on the cone $S_d^+$ of
symmetric positive semi-definite $d\times d$ matrices is
 defined (whenever it exists) by its Laplace transform
\begin{equation}\label{FLT Mayerhofer Wishart}
\mathcal L (\gamma(p;\sigma))(u)= \left(\det(I+\sigma
u)\right)^{-p},\quad u\in - S_d^+,
\end{equation}
were $p>0$ denotes its shape parameter and $\sigma\in S_d^{+}$ is
the scale parameter. In the non-degenerate case where $\sigma$ is
invertible, and for a discrete set of shape parameters, these
distributions have been introduced in 1928 by Wishart \cite{Wishart}
as sums of "squares'' of centered multivariate normal distributions
(quite similarly to the the chi-squared distributions in dimension
one).

In 1937, L\'evy \cite{Levy} showed that $\gamma(p;\sigma)$ on
$S_2^+$ is not infinitely divisible for invertible $\sigma$, which
means that for some sequence of shape parameters $p_k\downarrow 0$,
$\gamma(p_k;\sigma)$ cannot exist. Gindikin \cite{Gindikin},
Shanbhag \cite{Shanbhag} and Peddada \& Richards
\cite{PeddadaRichards91}\footnote{Contrary to
\cite{PeddadaRichards91} we exclude the point mass at zero, i.e. the
Gindikin ensemble does not contain $0$. That's why we have chosen
$p>0$.} subsequently showed: \\\\{\bf Theorem G.  }{\it For
non-degenerate $\sigma$ the following are equivalent:
\begin{enumerate} \item The right side of \eqref{FLT Mayerhofer
Wishart} is the Laplace transform of a probability measure. \item
$p$ belongs to the Gindikin ensemble
\[
\Lambda_d=\left\{\frac{j}{2},\quad j=1,2,\dots,
d-2\right\}\cup\left[\frac{d-1}{2},\infty\right).
\]
\end{enumerate}
}
Aim of the present note is to extend this characterization by also
allowing for degenerate $\sigma\in S_d^+$, and, to determine all
Wishart distributions which are infinitely divisible.

It is worth noting that -- in view of the right side of \eqref{FLT
Mayerhofer Wishart} and L\'evy's continuity theorem -- the maximal
parameter domain
\[
\Theta_d:=\{(p,\sigma)\in \mathbb R_+\times S_d^+\mid
\Gamma(p;\sigma)  \textrm{ is a probability measure}\}
\]
is closed, and that only for invertible scale parameters the issues
of existence and infinite divisibility have been settled in the
literature (where one easily follows  from the other). Hence the
following two statements answer very natural question on the Wishart
family.
\begin{theorem}\label{th: maintheorem}
Let $d\in\mathbb N$, $p\geq 0$ and $\sigma \in S_d^{+}$ with
$\rank(\sigma)=r$. The following are equivalent:
\begin{enumerate}
\item \label{char 1} The right side of \eqref{FLT Mayerhofer Wishart} is the Laplace transform of a non-trivial probability
measure $\Gamma(p;\sigma)$ on $S_d^+$.
\item \label{char 2} $p\in \Lambda_r$.
\end{enumerate}
\end{theorem}

\begin{theorem}\label{th: side result}
Let $d\in\mathbb N$, $p\geq 0$ and $\sigma \in S_d^{+}$. The
following are equivalent:
\begin{enumerate}
\item \label{char 3} $\Gamma(p;\sigma)$ is infinitely divisible.
\item \label{char 4} $\rank(\sigma)=1$.
\end{enumerate}
\end{theorem}

\begin{remark}\rm
\begin{itemize}
\item It should be pointed out that Theorem \ref{th: maintheorem} does not
follow from the classification of positive Riesz distributions, when
$\rank(\sigma)<d$, see, e.g., \cite[Theorem VII.3.2]{FK04}. In fact
let us consider the special case, where
$\sigma=\diag(1,0,\dots,0)\in S_d^+$, and denote by $\pi_{d\to r}$
the projection onto the $r$--th subminors of $S_d^+$, that is
\begin{equation}\label{proj p d to r}
\pi_{d\to r}:\, S_d^+\rightarrow S_r^+,\quad \pi_{d\to
r}(a)=(a_{ij})_{1\leq i,j\leq r}.
\end{equation}
Then the right side of \eqref{FLT Mayerhofer Wishart} takes the form
\begin{equation}\label{flt counter}
(\det (\pi_{d\to r} (1+u)))^{-p}=(1+u_{11})^{-p},
\end{equation}
which the Laplace transform of an infinitely divisible probability
measure on $S_d^+$ due to the theorems \ref{th:
maintheorem}--\ref{th: side result}. However, the "corresponding''
Riesz distribution \footnote{This is a member of the natural
exponential family generated by the standard Riesz distribution with
Laplace transform $(\det (\pi_{d\to r} (u^{-1})))^{p}$} (which is
also infinitely divisible) has Laplace transform
\[
(\det(\pi_{d\to r} ((I+u)^{-1})))^{p}=\left(\frac{1+
u_{22}}{\det(I+u)}\right)^p,
\]
which obviously differs from eq.~\eqref{flt counter}. Nevertheless,
both functions are characteristic functions of probability measures
on the sub-cones of  positive matrices of rank $1$.
\item It is possible to generalize the Theorems of
this paper to symmetric cones. In that more general setting the set
$\Lambda$ must be replaced by the so-called Wallach set, see, e.g.,
\cite[Theorem VII.3.1]{FK04}. We avoid this setting to make the
presentation short and accessible to a larger group of readers.
\end{itemize}

\end{remark}

Notation: $I_d$ is the $d\times d$ unit matrix, and if there arises
no confusion, we simply write $I$. $\det$ and $\tr$ are determinant
and trace operator, and $\rank$ denotes the matrix rank. For
matrices $A,B$ of $r\times r$ and $t\times t$ dimension,
$\diag(A,B)$ denotes the corresponding block-diagonal $(r+t)\times
(r+t)$ matrix.

\begin{proof}[Proof of Theorem \ref{th: maintheorem}]
Let $U$ be an orthogonal matrix such that $U\sigma
U^\top=\diag(D,0,\dots,0)$ where $D=\diag(\sigma_1,\dots,\sigma_r)$.
We introduce the linear automorphism $g_U$ as the map
\[
g_U:\,S_d^+\rightarrow S_d^+,\quad g_U(\xi):=U\xi U^\top.
\]

Proof of \ref{char 1}$\Rightarrow$ \ref{char 2}: We denote by
$\Gamma_*$ the push-forward measure of $\Gamma(p;\sigma)$ under
$g_U$, which means that for Borel sets $A\in S_d^+$ we have
$\Gamma_*(A)=\Gamma(p;\sigma)(g_U^{-1}(A))=\Gamma(p;\sigma)(U^\top A
U)$. By \cite[eq.~(1.3.4)]{GuptaNagar01} the functional determinant
of $g_U$ equals $1$, hence for all $s\in -S_d^+$ we have
\begin{align}\nonumber
\mathcal
L(\Gamma_*)(s)&:=\int_{S_d^+}e^{\tr(s\xi)}\Gamma^*(d\xi)=\int_{S_d^+}e^{\tr(s\xi)}\Gamma(p;\sigma)(d(U^\top\xi
U))\\\nonumber &=\int_{S_d^+}e^{\tr(s U\eta
U^\top)}\Gamma(p;\sigma)(d\eta)=\int_{S_d^+}e^{\tr(U^\top s U \eta
)}\Gamma(p;\sigma)(d\eta)\\\nonumber &=\det(I+ \sigma U^\top s
U)^{-p}=\det(I+ U\sigma U^\top s)^{-p}\\\label{eq: pullover}
&=\det(I_d+ \diag(D,0) s)^{-p}=\det(I_r+ D s)^{-p}.
\end{align}
By Theorem G, we have that there exists a probability measure
$\Gamma(p;D)$ on $S_d^+$ with Laplace transform
\[
\det(I_r+ Du )^{-p},\quad u\in -S_r^+.
\]
We may therefore conclude that $\Gamma_*$ equals
$\Gamma(p;D)^{*,\pi_{d\to r}}$, the pullback of $\Gamma(p;D)$ under
$\pi_{d\to r}$ (see the defining equation \eqref{proj p d to r}). In
fact, for all $s\in -S_d^+$ we obtain
\[
\mathcal L(\Gamma(p;D)^{*,\pi_{d\to r}})(s)=
\int_{S_d^+}e^{\tr(\pi_{d\to r}(s)\xi)}\Gamma(p;D)(d\xi),
\]
which equals eq.~\eqref{eq: pullover}. But since
$\diag(\sigma_1,\dots,\sigma_r)$ is of full rank $r$, we have by
Theorem G that $p\in \Lambda_r$, and we have shown that \ref{char 2}
holds.

Proof of \ref{char 2}$\Rightarrow$ \ref{char 1}: We reverse the
preceding arguments: By Theorem G, there exists a probability
measure $\Gamma(p,D)$ on $S_r^+$  for $p\in\Lambda_r$. Let
$\pi^{r\to d}$ be the embedding
\[
\pi^{r\to d}: S_r^+\rightarrow S_d^+,\quad \pi^{r\to d}(A):=\diag
(A,0).
\]
Then the Laplace transform of the push-forward
$\Gamma(p,D)^{\pi^{r\to d}}_*$ of $\Gamma(p,D)$ under $\pi^{r\to d}$
equals
\[
\mathcal L(\Gamma(p,D)^{\pi^{r\to d}}_*)=\det(I_d+
\diag(D,0,\dots,0)\pi_{d\to r}(s))^{-p},
\]
Hence by eq.~\eqref{eq: pullover} the push-forward of
$\Gamma(p,D)^{\pi^{r\to d}}_*$ under
$g_{U^{-1}}=g_U^{-1}=g_{U^\top}$ (hence given by
$g_U^{-1}(\xi)=U^\top\xi U$) equals $\Gamma(p,\sigma)$, which proves
\ref{char 1}.
\end{proof}

\begin{proof}[Proof of Theorem \ref{th: side result}]
\ref{char 3}$\Rightarrow$ \ref{char 4}: By infinite divisibility and
the right side of eq.~\eqref{FLT Mayerhofer Wishart} we see that
$\Gamma(p/n;\sigma)$ is a probability measure for each $n\geq 1$.
But by definition, $\Gamma_d\subseteq [1/2,\infty)$ if and only if
$d>1$. Hence by Theorem \ref{th: maintheorem} we have that
$\rank(\sigma)=1$.

The converse implication \ref{char 4}$\Rightarrow $\ref{char 3} may
be proved in a similar way: By Theorem \ref{th: maintheorem} and in
view of the fact that $\Gamma_1=(0,\infty)$, we have that
$\Gamma(p/n;\sigma)$ is a probability measure for all $n\geq 1$.
Therefore by the right side of eq.~\eqref{FLT Mayerhofer Wishart} we
see that $\Gamma(p;\sigma)$ is infinitely divisible.
\end{proof}


\end{document}